\documentclass[letterpaper]{article}
\usepackage{geometry}
\pdfoutput=1 
\usepackage{amsmath,amssymb,amsthm,times,dsfont,framed, graphicx, textcomp,mathtools}
\usepackage{tikz, braids}
\usepackage[hidelinks]{hyperref}
\usepackage{listings}
\usepackage{xcolor}
 
\definecolor{background}{rgb}{0.94,0.94,0.92}
 
\lstdefinestyle{mystyle}{
    backgroundcolor=\color{background},   
    commentstyle=,
    keywordstyle=,
    numberstyle=,
    stringstyle=,
    basicstyle=\ttfamily\footnotesize,
    breakatwhitespace=false,         
    breaklines=true,                 
    captionpos=b,                    
    keepspaces=true,                 
    numbers=left,                    
    numbersep=5pt,                  
    showspaces=false,                
    showstringspaces=false,
    showtabs=false,                  
    tabsize=2
}
 
\newtheorem{alg}{Algorithm}
\newtheorem{defn}{Definition}
\newtheorem{prop}{Proposition}
\newtheorem{lem}{Lemma}
\newtheorem{thm}{Theorem}
\newtheorem{examp}{Example}

\newtheorem*{alg*}{Algorithm}
\newtheorem*{defn*}{Definition}
\newtheorem*{prop*}{Proposition}
\newtheorem*{lem*}{Lemma}
\newtheorem*{thm*}{Theorem}

\newcommand{\btikz}{\begin{tikzpicture}}
\newcommand{\etikz}{\end{tikzpicture}}
\newcommand{\CC}{\mathcal{C}}
\newcommand{\DD}{\mathcal{D}}
\newcommand{\ZZ}{\mathbb{Z}}
\newcommand{\ta}{\tau}

\newcommand{\un}{\underline}
\DeclareMathOperator{\Aut}{Aut}
\DeclareMathOperator{\Irr}{Irr}
\DeclareMathOperator{\id}{id}
\DeclareMathOperator{\BrPic}{BrPic}
\DeclareMathOperator{\Fib}{Fib}
\DeclareMathOperator{\Ising}{Ising}
\DeclareMathOperator{\Obj}{Obj}
\makeatletter
\newcommand{\thickhline}{%
    \noalign {\ifnum 0=`}\fi \hrule height 2pt
    \futurelet \reserved@a \@xhline
}
\begin{document}
\title{Fusion rules for permutation extensions of modular tensor categories}
\author{Colleen Delaney}

\maketitle

\begin{abstract}We give a construction and algorithmic description of the fusion ring of permutation extensions of an arbitrary modular tensor category using a combinatorial approach inspired by the physics of anyons and symmetry defects in bosonic topological phases of matter. The definition is illustrated with examples, namely bilayer symmetry defects and $S_3$-extensions of small modular tensor categories like the Ising and Fibonacci theories. An implementation of the fusion algorithm is provided in the form of a Mathematica package. We introduce the notions of confinement and deconfinement of anyons and defects, respectively, which develop the tools to generalize our approach to more general fusion rings of $G$-crossed extensions. 
\end{abstract}
\section{Introduction}
\label{sec:intro}
Modular tensor categories and their module categories provide an algebraic framework to describe anyons, boundaries, and defects in (2+1)D bosonic topological phases of matter. A special case is when a modular tensor category (MTC) with a $G$-action admits a $G$-crossed graded extension by a family of invertible bimodule categories. 

When $G$ acts not just as a group on the decategorified part of an MTC, namely its fusion ring, but acts on $\CC$ by braided monoidal autoequivalences, then the simple objects in the invertible bimodule categories $\CC_g$ describe point-like ``twist defects'' at the end of invertible domain walls that can be manifest by adding terms to the Hamiltonian of the anyon theory. Depending on whether the $\CC_g$ form a fusion category extending $\CC$, the extension is interpreted as symmetry-enriched topological (SET) order or anomalous SET order. In the latter case, there exists a $G$-crossed braided fusion category which can be interpreted as the algebraic theory of anyons and symmetry defects in (2+1)D bosonic SET order. 

The obstruction theory determining the existence and subsequent classification of $G$xBFCs was given in by Etingof, Nikshych, and Ostrik in \cite{ENO}: the obstructions to lifting a group action on a braided fusion category $\CC$ to a categorical group action and the obstructions to the pentagon axioms needed for the $\CC_g$ to form a $G$-crossed braided fusion category (GxBFC) are measured by cohomology classes in $H^3(G;A)$ and  $H^4(G;U(1))$, respectively. The inequivalent $G$xBFCs then form a $H^2(G;A) \times H^3(G;U(1)$-torsor. 

While \cite{ENO} provides a classification of SET order, it is still worthwhile to give an explicit construction of $G$xBFCs in terms of MTCs and suitable categorical group actions by $G$ both from the point of view of computational physics and abstract quantum algebra. 

A simple but important example is the case of permutation symmetry of multilayer topological order, corresponding to $S_n$-crossed braided extensions of Deligne product MTCs of the form $\CC^{\boxtimes n}$. While its counterpart in conformal field theory (permutation orbifolding) is well understood, and algebraic data for permutation extensions has been described in the language of TQFT, for practitioners of condensed matter theory it is desirable to have a construction purely in terms of MTCs. 

In this paper we develop an elementary approach to modeling the fusion rings of GxBFCs in order to understand the fusion rules satisfied by topological charges and symmetry defects in SET phases of matter and apply it to construct the fusion rings for permutation extensions $\left(\CC^{\boxtimes n}\right)^{\times}_{S_n}$ of MTCs.

We define an $H^2(S_n;A^{\boxtimes n})$-torsor of $S_n$-crossed ring extensions of the fusion ring of a Deligne product category $\CC^{\boxtimes n}$ for arbitrary modular tensor categories $\CC$. The main theorem is an explicit construction of the possible fusion rules for $(\CC^{\boxtimes n})^{\times}_{S_n}$ in terms of the fusion rules for $\CC$ and a choice of 2-cocycle valued in $A^{\boxtimes n}$, the group of abelian anyon types in $\CC^{\boxtimes n}$. 

\begin{thm*}[Fusion rings of permutation extensions of MTCs]  The fusion rings of  $S_n$-crossed braided extensions of $\CC^{\boxtimes n}$ are given by the permutation defect fusion rings $(C^{\times}_{S_n},\otimes_{\omega})$, which can be computed from the data $(C, n, \omega)$. 
\end{thm*}

Our construction gives a short algorithm to compute the fusion rules.   \begin{alg*}[Permutation defect fusion algorithm]
  \label{alg:fusion}
  The fusion product of two permutation defects $X^{\sigma}_{\vec{a}}$ and $X^{\tau}_{\vec{b}}$ can be computed as follows. 
  \begin{enumerate}
  \item $\sigma$- and $\tau$-deconfinement: \\ Strip the topological charges from the defects and twist with the abelian anyon $\omega(\sigma,\tau)$.
  \item Transposition defect annihilation:\\ Compute the fusion product of the bare $\sigma$- and $\tau$-defects, for every pair of indices $(ij)$ permuted by both $\sigma$ and $\tau$, pulling out a factor of $$\bigoplus_{c \in \Irr(\CC)} \cdots \underbrace{c}_{i} \cdots \underbrace{c^*}_{j} \cdots.$$
  
  \item $\sigma\tau$-confinement:\\ Confine the product of the objects from Step 1 and Step 2 with the bare $\sigma\tau$-defect.
  \end{enumerate}
  \end{alg*}
  
An implementation in the form of a Mathematica package \textbf{PermutationDefectFusion.m} is provided at the author's website, see Appendix \ref{app:code} for a code sample showing the main function. 

We cast the statements and proofs of our results purely in terms of classical abstract algebra - rings, modules, and group actions - to demonstrate the principle that classical things can be understood through classical means. On the other hand, higher data requires higher data. But the benefit of the topological phase-inspired approach is not just that it is straightforward to construct the fusion rings, it becomes transparent how to categorify the fusion rings, which we discuss in upcoming work \cite{DS} and provides alternate categorical proofs of the results herein.



\subsection{Related work and acknowledgments}
\label{sec:acknowl}
The fusion rules for permutation defects were first understood in terms of modular functors in \cite{BarmeierSchweigert}. The approach taken here was independent and uses the data $(C, n, \omega)$ consisting of a fusion ring $C$ of an MTC $\CC$, a choice of $n$, and a choice of $2$-cocycle. 

In contemporaneous work \cite{BJ}, Bischoff and Jones give a more general algorithm for fusion rules for spherical $G$-crossed braided extensions of arbitrary fusion categories $\CC$ using different methods. They apply it to derive formulas for the fusion rules of maximal cyclic subextensions $\left ( \CC^{\boxtimes n} \right)^{\times}_{\ZZ/n\ZZ}$ of permutations $\left ( \CC^{\boxtimes n} \right)^{\times}_{S_n}$, giving an alternate description to the one that follows from our Section 4. 

The author thanks Christoph Schweigert for bringing his and his collaborators' work to their attention and to Corey Jones and Marcel Bischoff for coordinating the publication of related results. Thanks as well to Eric Samperton for helpful advice with an earlier version of this manuscript. The perspective developed here was heavily influenced by the author's time at Microsoft Station Q and thanks Zhenghan Wang, Mike Freedman, and Parsa Bonderson.
 
 \subsection{Overview of contents}
 \label{sec:overview}
We begin in Section \ref{sec:prelim} with a brief review of the algebraic theory of anyons and symmetry defects and establish the notions needed to work with them at the level of their fusion rings. In Section \ref{sec:theory} we introduce the terminology and tools like $g$-confinement and $g$-deconfinement that are then applied in Section \ref{sec:permdefring} to construct the permutation defect fusion ring, i.e. the fusion ring of the GxBFCs $\left ( \CC^{\boxtimes n} \right)^{\times}_{S_n}$. Examples when $n=2$ and $n=3$ are given in Section 5 to demonstrate the defect fusion algorithm and illustrate the main features of the general theory from Section \ref{sec:theory}.

We conclude our discussion in Section 6 with a brief comment on generalizations of our approach for general $G$xBFCs and applications.

\section{Preliminaries}
\label{sec:prelim}
Although our goal is to understand permutation extensions of MTCs as categories, the approach used here to compute the $S_n$-crossed fusion ring doesn't require any higher data. The fusion ring and its physical interpretation can be understood in terms of elementary abstract algebra, requiring only knowledge of the symmetric group $S_n$, $\mathbb{Z}_+$-based rings (in the sense of \cite{EGNO}), and second group cohomology. For this reason we present our results in a ``decategorified'' way and refer the reader to other sources for detailed definitions of MTCs and GxBFCs and their interpretation as algebraic theories of anyons and symmetry defects in (2+1)D topological phases of matter see \cite{BBCW, Dissertation}. 

That being said, we will still use the notation $\otimes$ and $\oplus$ for multiplication and addition in the fusion ring. Especially in Section \ref{sec:permdefring} we abuse notation and write objects to mean their isomorphism class, conflating anyons and defects with their types. 

These choices have the effect of making the proofs of our results elementary, if a bit inelegant. However, the techniques we use here to construct $G$-crossed braided fusion rings readily suggest the form of their categorification, which we construct in an upcoming sequel \cite{DS} and which provides a categorical proof of the following results.

\subsection{Algebraic theory of anyons and topological charge fusion}
\label{sec:anyonfusion}
Although strictly speaking the interpretation of MTCs as topological order is only for unitary MTCs and our results are stated for more general braided fusion categories, we will freely use the physics terminology for UMTCs throughout. 

\begin{defn*} An anyon is a simple object in a unitary modular tensor category $\CC$. An anyon type, or topological charge, is the isomorphism class of an anyon. 
\end{defn*}

For our purposes it will be enough to know that the topological charges form a unital, commutative, based $\ZZ_+$-ring, i.e. a (braided) fusion ring. We will consistently use calligraphic fonts for fusion categories $\CC$ and standard font $C$ to denote fusion rings. Passage from the calligraphic font to standard font $(\CC \to C)$ indicates the appropriate decategorification. For groups like $G$ and $\Aut(C)$ their categorical groups are indicated by $\un{G}$ and $\un{\Aut(C)}$.

\begin{defn*}[Unital based $\ZZ_+$-ring \cite{EGNO}.]
\label{def:fusring}
Let $C$ be a ring which is free as a $\ZZ$-module. A $\ZZ_+$-basis of a $C$ is a set of elements $B=\{b_i\}_{i \in I}$ such that $b_ib_j=\sum_{k\in I} N^k_{ij} b_k$, where $N^k_{ij} \in \ZZ_+$. 

A unital based $\ZZ_+$-ring is a ring with a fixed $\ZZ_+$-basis $B$ such that 
\begin{enumerate}
\item $1 \in B$
\item there exists an involution $i \mapsto i^*$ of $I$ such that the induced map $a \mapsto a^*$ is an anti-involution of $C$ and whenever $b_ib_j = \sum N^{k}_{ij} b_k$, $N^1_{ij} = \begin{cases} 1 & i=j^* \\ 0 & i \ne j^* 
\end{cases}$.
\end{enumerate}
\end{defn*}

Given a basis of topological charges (or defect types, see Definition \ref{def:gxbfr}) equipped with the involution that sends every object to its dual, the data of the fusion ring is encoded by the fusion coefficients $N^{ab}_c$. The abelian topological charges -- those $a \in \mathcal{L}$ such that for all $b \in \mathcal{L}$ there exists a unique $c \in \mathcal{L}$ with $N^{ab}_c \ne 0$ -- form an abelian group under fusion, denoted by $A$.

\begin{defn*}[Group action on fusion ring] 
Let $G$ be a group and $C$ a braided fusion ring with fixed basis. A $G$-action on $C$ is a homomorphism
$$G \longrightarrow \Aut(C)$$
where $\Aut(C)$ is the group of involution-preserving ring isomorphisms of $C$. 
\end{defn*}
A $G$-action on a fusion ring induces the symmetry $N^{ab}_c=N^{g\cdot a, g\cdot b}_{g \cdot c}$ of the fusion coefficients. The obstruction that measures whether a family of $G$-graded $C$-bimodules coming from a group action on $C$ under the isomorphism $\BrPic(C) \cong \Aut(C)$ form a fusion ring is given by a cohomology class in $H^3(G;A)$ \cite{ENO}. 

\subsection{Algebraic theory of symmetry defects $G$-crossed braided fusion rings}
\label{sec:defectfusion}

When the $H^3(G;A)$ obstruction vanishes, the extension fusion ring has the structure of a $G$-crossed braided ($G$-crossed commutative) fusion ring.

\begin{defn*}[G-crossed braided fusion ring]
\label{def:gxbfr} Let $C$ be a fusion ring of an MTC with a $G$-action. A $G$-crossed braided fusion ring $C^{\times}_G$ is a fusion ring which admits a $G$-grading by C-bimodules $C_g$ where $C_{\id}=C$, together with a $G$-action on $C^{\times}_G$ such that 
\begin{equation}
X_g \otimes Y = g \cdot Y \otimes X_g
\end{equation}
for all $X_g \in C_g$ and $Y \in C^{\times}_G$. 

\end{defn*}

An additional $H^4(G;U(1))$ obstruction to the existence of $G$xBFCs measures the failure for the invertible $\CC$-bimodule categories $\CC_g$ to satisfy the pentagon axioms and form a fusion category \cite{ENO}. In other words, it determines whether a $G$-crossed braided ring extension of $C$ lifts to a $G$-crossed braided extension of $\CC$.

\begin{defn} Given a categorical group action of $G$ on an MTC $\CC$, a (group) symmetry defect is a simple object in an invertible $\CC$-bimodule category $\CC_g$ under the equivalence $\un{\BrPic(\CC)}\simeq \un{\Aut^{br}_{\otimes}(\CC)}$. A symmetry defect type is its isomorphism class.
\end{defn}

When the $H^3$ obstruction vanishes but the $H^4$ doesn't, the $G$-graded extension of $\CC$ by bimodule categories is interpreted as \emph{anomalous} $G$-symmetry enriched topological order, in the sense that the fusion cannot be realized by point-like objects in a strictly (2+1)D system, but can potentially be realized as a 2-dimensional slice of some $(N+1)$D system. 
 
In this case the decategorified part of the anomalous $G$-extension is still a fusion ring - just a fusion ring which cannot be lifted to a $G$-crossed braided fusion category. While in the case of permutation extensions the $H^4(G;U(1))$ obstruction does vanish \cite{GJ}, the tools we develop should be applicable to the case of anomalous SET order and thus we keep our discussion slightly more general.

\subsection{Permutation extensions of MTCs}
\label{sec:permextensions}
Let $P$ be the monoidal functor that acts strictly on $\CC^{\boxtimes n}$ by permutations 

$$\begin{array}{cccc} P: &\un{S_n} &\longrightarrow& \un{\Aut_{\otimes}^{br}(\CC^{\boxtimes n})} \\\\
& id & \mapsto& \id: \CC^{\boxtimes n} \to \CC^{\boxtimes n} \\\\
& \sigma & \mapsto& 
\begin{array}{cccc} 
T_{\sigma}:& \CC^{\boxtimes n} &\to & \CC^{\boxtimes n} \\
& \boxtimes_i X_i & \mapsto & \boxtimes_i X_{\sigma(i)} \\
 & \boxtimes_i f_i &\mapsto& \boxtimes_i f_{\sigma(i)}  \end{array}\\
\end{array}$$
so that the tensorators $U_{\sigma}:T_{\sigma}(X \otimes Y) \to T_{\sigma}(X) \otimes T_{\sigma}(Y)$ and compositors $\eta_X: (T_{\rho}\circ T_{\sigma}) (X) \to T_{\rho\sigma}(X)$ are the identity isomorphisms for all $\rho,\sigma \in S_n$ and $X,Y \in \Obj(\CC^{\boxtimes n})$.

For these (untwisted) permutation actions on MTCs it is known that the $H^4(G;U(1))$ obstruction vanishes:

\begin{thm*}[\cite{GJ}]
\label{GJ}
The $H^4(S_n;U(1))$ obstruction to $S_n$-extensions of $\CC^{\boxtimes n}$ vanish for the categorical permutation group action $P: \un{S_n} \to \un{\Aut_{\otimes}^{br}(\CC^{\boxtimes n})}$ and the equivalence classes of $S_n$-extensions form a torsor over $H^3(G;U(1))$. 
\end{thm*}

The permutation symmetry models a global unitary on-site symmetry of \emph{multi-layer topological order} given by $n$ layers of topological order $\CC$ \cite{BBCW}. 

$$\begin{tikzpicture}[scale=.5,baseline=-32.5]
\draw[thick,fill=black!20] (.5,.5)--(0,0)--(4,0)--(4.5,.5) node[below] {$\CC$}--(.5,.5);
\draw[thick,fill=black!20] (.5,-.5)--(0,-1)--(4,-1)--(4.5,-.5) node[below] {$\CC$}--(.5,-.5);
\draw (2.25,-1.5) node {\small $\vdots$};
\draw[thick,fill=black!20] (.5,-2.5)--(0,-3)--(4,-3)--(4.5,-2.5) node[below] {$\CC$} --(.5,-2.5);
\draw (-.5,-1.5) node[left] {$S_n \curvearrowright$} ;
\draw[thick, ->] (5,-1.5)--(7,-1.5) node[right] {$ \left ( \CC^{\boxtimes n } \right)^{\times}_{S_n}=\bigoplus_{\sigma \in S_n} \CC_{\sigma}$};
\end{tikzpicture}$$

While we been interpreting the factors in the Deligne product as spatially separated layers, the permutation symmetry is technically on-site because $\CC^{\boxtimes n}$ can also be interpreted a monolayer topological order. This is what allows us to study the spatial symmetry using the techniques developed for on-site symmetries in \cite{BBCW}.

Prior to the development of $G$xBFCs as the algebraic theory of symmetry defects physicists studied permutation defects in (2+1)D TPM under the guise of \emph{genons} \cite{BJQ}, so called because of the way they effectively couple layers of topological phases to create nontrivial topology \cite{BBCW, BJQ}.

The ability of genons to entangle the layers comes from the $S_n$-crossed braiding, whereby exchange with defects transports (monolayer) anyons between layers. 

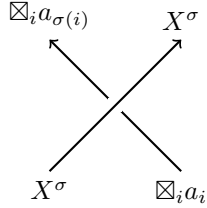
\begin{figure}[h!]
$$\btikz[scale=.35,yshift=8cm] 
\draw[thick,->] (5,0)node[below] {$\boxtimes_i a_i$}--(0,5) node[above] {$\boxtimes_i a_{\sigma(i)} $};
\fill[white] (2.25,2.25) rectangle (2.75,2.75);
\draw[thick,->] (0,0) node[below] {$X^{\sigma}$}--(5,5) node[above] { $X^{\sigma}$};
\etikz
$$
\end{figure}

\section{Confinement, deconfinement, and fusion rules between anyons and $g$-defects}
\label{sec:theory}
Our model of symmetry defect fusion is based on a parametrization of defect types by fixed points. 
\begin{eqnarray} \Irr(\CC_{g}) &=& \{ X^{g}_{f} \mid f \in \Irr(\CC) \text{ with } g\cdot f \cong f \}.\end{eqnarray} 
One benefit is that when $\sigma=\id$ one recovers the anyons in $\CC=\CC_{\id}$ as the fixed points under the action by the identity element of $G$. Thus we can think of anyons as trivial symmetry defects, although in what follows we reserve the term ``symmetry defect'' to indicate $g\ne \id$. 
It is also worth mentioning that the defects within each sector inherit an ordering from an ordering of $\CC$.

\begin{defn}[Defect charge] Let $X^{g}_{f} \in \CC_g$ be a symmetry defect. Then $f$ is called the topological charge of the defect $X^g_f$. 
\end{defn}
Since the isomorphism class of the monoidal unit $1$ is always fixed by a $G$-action, each sector has a distinguished object $1$ with vacuum charge. Following \cite{BBCW}, we make the following definition.
\begin{defn}[Bare defect] A symmetry defect with vacuum charge $X^{g}_{1}$ is called a bare defect. 
\end{defn}

\subsection{Confined versus deconfined objects in topological phases}
Anyons, being intrinsic quasiparticle excitations corresponding to the ground state of a Hamiltonian which can be moved by local operators without additional energy, are said to be \emph{deconfined}.

On the other hand, point-like symmetry defects are not finite-energy excitations. They are extrinsic in the sense that a Hamiltonian realization of a topological phase enriched with symmetry needs additional terms added in order for it to have excitations which correspond to the symmetry defects \cite{BBCW}. Moreover, the energy needed to spatially separate defects grows differently than it does for anyons and they are said to be \emph{confined} as opposed to deconfined. 

We make the following definitions of $g$-confinement and $g$-deconfinement, borrowing the ideas from condensed matter theory.\footnote{These are at odds with the sense in which deconfinement is used in the context of anyon condensation, which is why we have made the dependence on the group element explicit.}

\subsection{$g$-confinements and $g$-deconfinement of anyons and $g$-defects}
\label{sec:gconfdeconf}
\begin{defn}[$g$-deconfinement]
\label{def:deconf} Let $X^g_f$ be a $g$-defect with charge $f$, and let $d(f) \in \Obj(\CC)$ be any object, not necessarily simple, with the property that
\begin{equation}
d(f) \otimes X^{g}_1\cong X^{g}_f.
\end{equation}
Then we say that $d(f)$ is a deconfinement, or write $d_g(f)$ is a $g$-deconfinement of the defect $X^{g}_f$. 
\end{defn}
In general a defect has multiple deconfinements.  

\begin{defn}[$g$-confinement]
\label{def:conf}
Let $a \in \Irr(\CC)$, $X^g_1 \in \Irr(\CC_g)$ a bare defect, and define \begin{equation}c(a) = \otimes_{k=1}^m \, g^k \cdot a \end{equation}
where $|g|=m$.  
Then the defect $X^g_{c(a)}$ is called the confinement of $a$. 
\end{defn}
Observe that $g\cdot c(a) \cong c(a)$, so that confinement is well-defined.

\section{Fusion rules for permutation defects}
\label{sec:permdefring}
Now we apply the ideas introduced in the previous section to construct $G$-crossed braided fusion rings directly in the case of $G=S_n$ acting on Deligne product MTCs $\CC^{\boxtimes n}$. 
The definition of the permutation defect fusion ring $(C^{\boxtimes n})^{\times}_{S_n}$ proceeds by specifying a free $\ZZ$-module on a basis of permutation defect types and a binary operation $\otimes$ that gives it the structure of a ``$G$-crossed commutative'' unital $\ZZ_+$ based ring.
\subsection{Model for multilayer anyons and permutation defects}
\label{sec:model}

We will consistently use $(C^{\boxtimes n},\otimes)$ to mean the fusion ring of $\CC^{\boxtimes n}$ with respect to a basis $$ \Irr(\CC^{\boxtimes n})=\{\boxtimes_i a_i | a_i \in \Irr(\CC), 1 \le i \le n\}.$$ 

Since the rank of each $\sigma$-sector in an $S_n$-extension is given by the number of fixed anyons under the action of $\sigma$, see for example \cite{Bischoff2018}, one can index the isomorphism classes of simple objects in $\CC_{\sigma}$ by 
\begin{eqnarray} \Irr(\CC_{\sigma}) &=& \{ X^{\sigma}_{\boxtimes_i a_i} \mid\boxtimes a_i \in \Irr(\CC^{\boxtimes n}) \text{ with } \boxtimes_i a_{\sigma(i)}\cong\boxtimes_i a_i \}.\end{eqnarray} 
Occasionally we use $\vec{a}:=\boxtimes_i a_i$ or supress the Deligne product and write $a_i$ to simplify notation. 

It will soon become apparent that this parametrization of symmetry defects is the key to constructing the fusion ring in a group-theoretical way, allowing one to construct the $S_n$-extension of $\CC^{\boxtimes n}$ from $S_2$-extensions of $\CC \boxtimes \CC$ in exactly the same way that $S_n$ is generated by its transpositions.

\begin{defn}{Permutation defect type basis}
Let $(C^{\boxtimes n})^{\times}_{S_n}$ be the free $\ZZ$-module on the set $$\Irr(\CC^{\boxtimes n}) \cup_{\sigma} \Irr(\CC_{\sigma})$$ which we will also write as $\bigcup_{\sigma \in S_n} \{X^{\sigma}_{\vec{a}}, \sigma\cdot\vec{a}=\vec{a}\}$. 
\end{defn}
The following sections endow $C^{\times}_{S_n}$ with the structure of a $G$-crossed braided fusion ring, see Definition \ref{def:gxbfr}.
 
\subsubsection{Overview of construction}

In Section \ref{sec:anydeffusion} we define $C^{\boxtimes n}$-bimodules $C_{\sigma}$ with respect to the bases $\{X^{\sigma}_{\vec{a}} \mid \sigma \cdot \vec{a}=\vec{a}, \vec{a} \in \Irr(\CC^{\boxtimes n}\}$ and in Section 4.3 define an $S_n$-action on the $\ZZ$-module $C^{\times}_{S_n}$ which is free on the basis $$\bigcup_{\sigma \in S_n} \{X^{\sigma}_{\vec{a}}, \sigma\cdot\vec{a}=\vec{a}\}.$$

Section 4.4 defines a multiplication $\otimes$ on $C^{\times}_{S_n}$ and show that this gives it the structure of an $S_n$-graded, $S_n$-crossed braided fusion ring  $C^{\times}_{S_n}=\bigoplus_{\sigma \in S_n} C_{\sigma}$ extending $C^{\boxtimes n}$. We call $(C^{\times}_{S_n},\otimes)$ the permutation defect fusion ring. 

In Section 4.5 we show that this fusion ring can be realized as a basepoint for an $H^2(S_n,A)$-torsor's worth of fusion rings $C^{\times}_{S_n,\omega}$. In the final Section 4.5 we conclude that the fusion rules for permutation extensions of MTCs can be constructed from the data $(N^{ab,c}, n, \omega)$. We give an algorithm to compute the fusion product of permutation defects and an implementation (see Appendix \ref{app:code}).

\subsubsection{Notation and Terminology}

We use standard permutation notation for $S_n$. For an arbitrary permutation $\sigma \in S_n$ we write its disjoint cycle decomposition as $\sigma=\sigma_1\sigma_2\cdots \sigma_s$, where each $\sigma_i$ is a cycle $(i_1 i_2 \cdots i_{m})$. 

By abuse of notation we conflate permutations with the subsets of $\{1,2,\ldots n\}$ that they act nontrivially on, and for example write $\sigma_i \cap \sigma_j=\emptyset$ to mean that the two permutations are disjoint and $i \in \sigma$ ($i \notin \sigma$) to indicate that a given $i \in \{1,2,\ldots n\}$ is (or isn't) permuted nontrivially by the action of $\sigma$. \\

Throughout we will use the physics terminology, whose corresponding mathematical meaning was given in the next table.
\begin{table}[h!]
\label{dict}
\begin{center}
\begin{tabular}{c|c}
\large \textbf{(2+1)D TPM} & \large \textbf{UMTC} \\ 
\hline
multi-layer topological order & Deligne product of UMTCs $\mathcal{C}^{\boxtimes n}$ \\ \hline 
multi-layer SET order & unitary $S_n$xBFC $\left(\mathcal{C}^{\boxtimes n}\right)^{\times}_{S_n} = \bigoplus_{\sigma \in S_n} \mathcal{C}_{\sigma}$ \\ \hline
(multilayer) anyon & $\vec{a}:=\boxtimes_i a_i$, where $a_i \in \Irr(\mathcal{C})$ \\ vacuum & iso. class of tensor unit $1:=1^{\boxtimes n}$\\\hline

monolayer anyon & $1 \boxtimes \cdots \boxtimes  1 \boxtimes a \boxtimes 1 \boxtimes \cdots 1$\\ \hline$g$-sector & invertible $\CC$-bimodule category $\CC_g$\\ \hline

transposition defect & simple object $X^{(ij)}_{\boxtimes_i f_i} \in \mathcal{C}_{(ij)}$\\ \hline

$m$-cycle defect & simple object $X^{(i_1i_2\cdots i_m)}_{\boxtimes_i f_i} \in \mathcal{C}_{(i_1i_2 \cdots i_m)}$ \\ \hline
permutation defect & simple object $X_{\boxtimes_i f_i}^{\sigma} \in \CC_{\sigma}$\\
bare defect & $X^{\sigma}_{1^{\boxtimes n}}$ \\
\end{tabular}
\end{center}
\caption{Mathematical definitions of the physics terminology we will use when discussing permutation-enriched topological order.}
\end{table}
\subsection{Fusion of multilayer anyons and permutation defects}
\label{sec:anydeffusion}
The multiplication $\otimes$ on $C^{\times}_{S_n}$ that we are about to define restricts to a commutative, associative binary operation on $C_{\id}=C^{\boxtimes n}$ and for this reason in any expression involving only products in $C^{\boxtimes n}$ we will freely commute elements and omit parenthesization without comment. 

First we restate Definitions \ref{def:deconf} and \ref{def:conf} in the case of permutation symmetry of $\CC^{\boxtimes n}$. 

\begin{defn}[$\sigma$-confinement]
\label{def:sigmaconf} Let $\vec{a} \in \Irr(\CC^{\boxtimes n})$ and write a disjoint cycle decomposition of $\sigma$ as $\sigma=\prod_j \sigma_j$. Then the $\sigma$-confinement map $$c_{\sigma}: C^{\boxtimes n} \to C^{\boxtimes n}$$ is defined on basis elements by 
\begin{equation}c_{\sigma}(\vec{a}) :=\boxtimes_i c_i \hspace{10pt} \text{  where  }  c_i = \begin{cases} a_i   & i \notin \sigma \\ \bigotimes_{k \in \sigma_j}\,  a_k& i \in \sigma_j \end{cases}
\end{equation}
and extended linearly to $C^{\boxtimes n}$. 
\end{defn}

\begin{prop}
\label{prop:confprop}The confinement map has the following properties.

\begin{enumerate} 
\item $c_{\sigma}(\vec{a} \otimes \vec{b}) = c_{\sigma}(\vec{a}) \otimes c_{\sigma}(\vec{b})$ 
\item $c_{\sigma}(\sigma^k \cdot \vec{a}) = c_{\sigma}(\vec{a})$ \hfill for all $\vec{a}, \vec{b} \in \Irr(\CC^{\boxtimes n})$,  $1 \le k \le |\sigma|$.

\end{enumerate}

\end{prop}
\begin{proof} Easy consequences of Definition \ref{def:sigmaconf}.
\end{proof}
These will be applied often in the proofs that follow. 

\begin{defn}[Permutation sectors]
Let $C_{\sigma}$ be the free $\ZZ$-module on the basis of $\sigma$-defect types $\Irr(\CC_{\sigma})$.
\end{defn}

The confinement map is the main ingredient in extending the fusion between anyons to an action on $\sigma$-defects. 
\begin{defn}[Anyon-defect fusion]
\label{def:conf} Let $\vec{a} \in \Irr(\CC)$ and $X^{\sigma}_{\vec{b}} \in \Irr(\CC_{\sigma})$. Then
we define a binary operation
$$\otimes: C^{\boxtimes n} \times C_{\sigma} \longrightarrow C_{\sigma} $$
on basis elements by
\begin{equation}
\vec{a} \otimes X^{\sigma}_{\vec{b}} := X^{\sigma}_{c_{\sigma}(\vec{a}) \otimes \vec{b}}.
\end{equation}
and identically for $C_{\sigma} \times C^{\boxtimes n} \longrightarrow C_{\sigma}$. 
\end{defn}

\begin{defn}[$\sigma$-deconfinement] 
\label{def:sigmadeconf}Let $X_{\vec{b}} \in \Irr(\CC_{\sigma})$, $\vec{a} \in \Irr(\CC^{\boxtimes n})$, and suppose they satisfy
\begin{eqnarray}
\vec{a} \otimes X^{\sigma}_{\vec{1}} = X^{\sigma}_{\vec{b}}.
\end{eqnarray}

Then we say that $\vec{a}$ is a (left) deconfinement of $X_b$, and define right deconfinements in the analogous way using the right action. 
\end{defn}
When we want to strip the topological charge from a defect by splitting off an anyon but without making a specific choice of deconfinement, we write
\begin{eqnarray}
X_{\vec{a}} &=& d(\vec{a}) \otimes X_{\vec{1}}.
\end{eqnarray}

We will see that the notion of deconfinement is central to our construction, as it allows us to intuit the way that the defect theory is built from the anyon theory.

\begin{lem}
\label{lem:deconfprop} Anyon-defect fusion is independent of choice of deconfinements. In other words, we can write

 \begin{equation} \vec{a} \otimes X^{\sigma}_{\vec{b}} = (\vec{a} \otimes d_{\sigma}(\vec{b})) \otimes X^{\sigma}_1
\end{equation} for any choice of deconfinement $d_{\sigma}(\vec{b})$. 
\end{lem}
\begin{proof}

We have 
\begin{eqnarray}
\vec{a} \otimes X^{\sigma}_{\vec{b}}  = &X^{\sigma}_{c_{\sigma}(\vec{a}) \otimes \vec{b}} \\
= &X^{\sigma}_{c_{\sigma}(\vec{a}) \otimes c_{\sigma}(d_{\sigma}(\vec{b}))} 
\end{eqnarray}
by the definitions of confinement/deconfinement, where $d_{\sigma}(\vec{b})$ is any $\sigma$-deconfinement. Then by Proposition \ref{prop:confprop},
\begin{eqnarray}
X^{\sigma}_{c_{\sigma}(\vec{a}) \otimes c_{\sigma}(d_{\sigma}(\vec{b}))}=& X^{\sigma}_{c_{\sigma}(\vec{a} \otimes d_{\sigma}(\vec{b}))}\\
=& (\vec{a} \otimes d_{\sigma}(\vec{b})) \otimes X^{\sigma}_1
\end{eqnarray}

\end{proof}

\begin{examp}[Fusion of monolayer anyons and bare defects]

The fusion between monolayer anyons $1^{\boxtimes i-1} \boxtimes a \boxtimes 1^{\boxtimes n-i}$ and bare defects $X^{\sigma}_{\vec{1}}$ is given by
\begin{eqnarray}
1^{\boxtimes i-1} \boxtimes a \boxtimes 1^{\boxtimes n-i} \otimes X_{\vec{1}} &= X_{\boxtimes b_j} & b_j =\begin{cases} 1 & \sigma(j)=j \\ a & \sigma(j) \ne j \end{cases} .
\end{eqnarray}

For example, when $n=2$, 
\begin{eqnarray} 1 \boxtimes a \otimes X_1 = a \boxtimes 1 \otimes X_1=  X_{aa} \end{eqnarray}
for all $a \in \Irr(\CC)$.

In words, fusing an anyon in the $i^{th}$ layer with the bare defect results is the defect with the topological charge label that has $a$ in every layer which is in the $\sigma$-orbit of $i$. 
\end{examp}
\begin{prop} $C_{\sigma}$ is a $(C^{\boxtimes n},C^{\boxtimes n})$-bimodule with respect to anyon-defect fusion.  
\label{prop:bimod}
\end{prop}
\begin{proof}Distributivity was built in to the definition of the confinement map and fusion, and it is immediate that $\vec{1} \otimes X^{\sigma}_{\vec{a}}=X^{\sigma}_{\vec{a}} \otimes \vec{1}$ for all $\vec{a}=\sigma\cdot\vec{a}$. The only axioms of a bimodule that need to be checked are left, right, and middle associativity. These are any easy consequence of Proposition \ref{prop:confprop} and commutativity of associativity of $\otimes$ in $C^{\boxtimes n}$. For example, 
\begin{eqnarray}
 (\vec{a} \otimes \vec{b}) \otimes X^{\sigma}_{\vec{c}} & = X_{c_{\sigma}(\vec{a} \otimes \vec{b}) \otimes \vec{c}} \\
 & =  X_{c_{\sigma}(\vec{a} )\otimes c_{\sigma}( \vec{b}) \otimes \vec{c}} \\
 & = \vec{a} \otimes X_{c_{\sigma}(\vec{b}) \otimes \vec{c}} \\
 & = \vec{a} \otimes ( \vec{b} \otimes X^{\sigma}_{\vec{c}}) \end{eqnarray}
The other cases are similar.
\end{proof}

\subsection{$S_n$-action on defects} 
So far we have a free $\ZZ$ module $C^{\times}_{S_n}$ which is an $S_n$-graded extension of $C^{\times}$ by bimodules:
$$C^{\times}_{S_n}= \bigoplus_{\sigma \in S_n} C_{\sigma}.$$

\begin{defn} Let $\rho \in S_n$. Then define a map
$$S_n \times C^{\times}_{S_n} \longrightarrow C^{\times}_{S_n}$$
on basis elements by
\begin{equation}
\rho \cdot X^{\sigma}_{\vec{a}}:= X^{\rho\sigma\rho^{-1}}_{\rho \cdot \vec{a}}.
\end{equation}
\end{defn}
One can check this is a well-defined action of $S_n$ on $C^{\times}_{S_n}$ extending the action of $S_n$ on $C^{\boxtimes n}$ due to associativity in $S_n$. 
\begin{prop}[$S_n$-crossed commutativity of topological charge and defect type fusion] 
\label{lem:permcrossed1}
$$X^{\sigma}_{\vec{a}} \otimes \vec{b} = \sigma \cdot \vec{b} \otimes X^{\sigma}_{\vec{a}}$$
\end{prop}
\begin{proof}
We have 
\begin{eqnarray}
X^{\sigma}_{\vec{a}} \otimes \vec{b} 
 =& (\vec{b} \otimes d_{\sigma}(\vec{a}))  \otimes X^{\sigma}_{\vec{1}}\\
 =& X^{\sigma}_{c_{\sigma}(\vec{b} \otimes d_{\sigma}(a))} \\
 =&  X^{\sigma}_{c_{\sigma}(\vec{b}) \otimes c_{\sigma}(d_{\sigma}(a))} \\
= &  X^{\sigma}_{c_{\sigma}(\sigma\cdot \vec{b}) \otimes \vec{a})}\\
= &\sigma\cdot \vec{b} \otimes  X^{\sigma}_{\vec{ a}}
\end{eqnarray}
By Lemma \ref{lem:deconfprop}, Definition \ref{def:conf}, and parts (1) and (2) of Proposition \ref{prop:confprop}.
\end{proof}

Later we will show that the multiplication on all of $C^{\times}_{S_n}$ is $S_n$-crossed commutative but Lemma \ref{lem:permcrossed1} will be helpful for showing said multiplication is associative. 
\subsection{Fusion of permutation defects}
Finally we are ready to define a product on all of $C^{\times}_{S_n}$. 
\begin{defn}[Defect fusion]
\label{def:defectfusion} Let $X^{\rho}_{\vec{a}}$ and $X^{\sigma}_{\vec{b}}$ be two symmetry defects. Then their (untwisted) product is given by 
\begin{equation}
X^{\rho}_{\vec{a}} \otimes X^{\sigma}_{\vec{b}} := \left( d_{\rho}(\vec{a}) \otimes d_{\sigma}(\vec{b}) \otimes \left ( \bigotimes_{(i_ki_l) \in \rho\sigma} \bigoplus_{c \in \Irr(\CC)}  1^{\boxtimes i_k-1} \boxtimes c \boxtimes 1^{i_l -i_k -1} \boxtimes c^* \boxtimes 1^{n-i_l}  \right) \right) \otimes X^{\rho\sigma}_{\vec{1}}.
\end{equation}
\end{defn}
It will become clear that the choice of ordering of factors in the product and left justification is arbitrary, but we will present our calculations this way consistently.  

The fusion product that results from the annihilation of transposition defects is an object with nice properties that will come in handy later. The next proposition says it can teleport anyons from layer to layer and is invariant under the layer-exchange symmetry:
 Let $\tau=(ij)$ be a transposition and write $$X^{(ij)}_{\vec{1}} \otimes X^{(ij)}_{\vec{1}}= \bigoplus_{c \in \Irr(\CC)}\cdots  \underbrace{c}_{i}  \cdots \underbrace{c^*}_j \cdots$$ as shorthand for the fusion product of bare transposition defects.
\begin{prop}[Properties of the transposition fusion product]
\label{prop:transpositionprod} The following equations hold.
\begin{enumerate}
\item $$\bigoplus_{c \in \Irr(\CC)}\cdots  \underbrace{a \otimes c}_{i}  \cdots \underbrace{c^*}_j \cdots =\bigoplus_{c \in \Irr(\CC)} \cdots   \underbrace{c }_{i} \cdots \underbrace{c^* \otimes a}_j  \cdots$$
for all $a \in \Irr(\CC)$.
\item $$(ij) \cdot \left ( \bigoplus_{c \in \Irr(\CC)}\cdots  \underbrace{c}_{i}  \cdots \underbrace{c^*}_j \cdots\right) =\bigoplus_{c \in \Irr(\CC)}\cdots  \underbrace{c}_{i}  \cdots \underbrace{c^*}_j \cdots$$
\end{enumerate}
\end{prop}
\begin{proof}
For (1), one can write
\begin{equation}
\begin{split}
\bigoplus_c  \cdots a \otimes c \cdots c^*  \cdots &= \bigoplus_{b,c}  \cdots N^{ac}_b \, b \cdots c^{*} \cdots \\
& = \bigoplus_{b,c} \cdots N^{b^*a}_{c^*}\,  b \cdots c^*  \cdots\\
&=  \bigoplus_{b,c}   \cdots b \cdots N^{b^*a}_{c^*}\,  c^* \cdots \\
&=  \bigoplus_b   \cdots b \cdots b^{*} \otimes a  \cdots
\end{split}
\end{equation}
Where $b$ and $c$ range over all of $\Irr(\CC)$ and we have used symmetries of the fusion coefficients that come from duality-induced isomorphisms of trivalent Hom spaces in $\CC$, see for example \cite{CBMS,EGNO}. Relabeling gives equation (1). Equation (2) follows immediately from the fact that the sum is over all of $\Irr(\CC)$ and relabeling. 
\end{proof}

The following lemma shows that it suffices to check associativity for bare transposition defects. 

\begin{lem}
\label{lem:permassoc}
Associativity of bare transposition defects implies associativity of all nontrivial permutation defects.
\end{lem}

\begin{proof}
Suppose
\begin{equation}
(X^{\tau_1}_{\vec{1}} \otimes X^{\tau_2}_{\vec{1}}) \otimes X^{\tau_3}_{\vec{1}} = X^{\tau_1}_{\vec{1}} \otimes ( X^{\tau_2}_{\vec{1}} \otimes X^{\tau_3}_{\vec{1}}).
\end{equation}
Then 
\begin{equation}
\begin{split} (X^{\tau_1}_{\vec{a}} \otimes X^{\tau_2}_{\vec{b}}) \otimes X^{\tau_3}_{\vec{c}} &= \left( d_1(\vec{a}) \otimes d_2(\vec{b}) \otimes \left (X^{\tau_1}_{\vec{1}} \otimes X^{\tau_2}_{\vec{1}}  \right) \right) \otimes \left ( d_3(\vec{c}) \otimes X^{\tau_3}_{\vec{1}} \right) \\
 &= \left( d_1(\vec{a}) \otimes  d_2(\vec{b}) \otimes  d_3(\vec{c}) \right ) \otimes \left( \left (X^{\tau_1}_{\vec{1}} \otimes X^{\tau_2}_{\vec{1}}  \right)  \otimes X^{\tau_3}_{\vec{1}} \right) \\
 &= \left( d_1(\vec{a}) \otimes  d_2(\vec{b}) \otimes  d_3(\vec{c}) \right ) \otimes \left( X^{\tau_1}_{\vec{1}} \otimes\left (X^{\tau_2}_{\vec{1}}   \otimes X^{\tau_3}_{\vec{1}} \right) \right) 
\\
&=  X^{\tau_1}_{\vec{a}} \otimes (X^{\tau_2}_{\vec{b}} \otimes X^{\tau_3}_{\vec{c}} ).
\end{split}
\end{equation}
Since transpositions generate $S_n$ associativity for arbitrary permutation defects follows immediately.

Thus it is enough to check associativity for bare transposition defects. 
\end{proof}

\begin{lem}[Bare transposition defect associativity]
\label{lem:bardefassoc}
\begin{equation}
\label{eq:baredefassoc}
(X^{\tau_1}_{\vec{1}} \otimes X^{\tau_2}_{\vec{1}}) \otimes X^{\tau_3}_{\vec{1}} = X^{\tau_1}_{\vec{1}} \otimes ( X^{\tau_2}_{\vec{1}} \otimes X^{\tau_3}_{\vec{1}}).
\end{equation}
\end{lem}
\begin{proof} There are several cases to check. We start with the border cases. 

(a) If $\tau_1, \tau_2, \tau_3$ are pairwise disjoint, then both sides of Equation \ref{eq:baredefassoc} simplify to $X^{\tau_1\tau_2\tau_3}_{\vec{1}}$ by part (1) of Proposition \ref{prop:disjointdefprod}. 

(b) If $\tau_1=\tau_2=\tau_3=\tau$, then both sides of Equation \ref{eq:baredefassoc} yield $X^{\tau}_{c_{\tau}(\oplus_{a \in \Irr(C)} \cdots a \cdots a^* \cdots)}$ by Definition \ref{def:defectfusion} .

(c) If $|\tau_1 \cap \tau_2 \cap \tau_3|=1$, Definition \ref{def:defectfusion}  gives $X^{\tau_1\tau_2\tau_3}_{\vec{1}}$. 
 
The remaining cases are similar and the details not particularly illuminating.
\end{proof}

\begin{lem} $S_n$-defect type fusion is associative, making $(C^{\times}_{S_n}, \otimes)$ into a ring.
\end{lem}
\begin{proof} 
Of course products involving three anyons are associative and associativity of triple-defect products follows from combining Lemmas \ref{lem:permassoc} and \ref{lem:bardefassoc}. By Propositions \ref{prop:bimod} and \ref{lem:permcrossed1}, products involving two anyons and one defect are associative. It remains only to check associativity for products involving one anyon and two defects.

We show
\begin{equation}
\label{eq:twodefprod} (\vec{a} \otimes X^{\rho}_{\vec{b}}) \otimes X^{\sigma}_{\vec{c}} = \vec{a} \otimes (X^{\rho}_{\vec{b}} \otimes X^{\sigma}_{\vec{c}})\end{equation} and claim that the other cases are similar. Note that it suffices to check the case where $\rho=\tau_1$ and $\sigma=\tau_2$ are transpositions, since together with Equation \ref{eq:twodefprod} Lemma \ref{lem:bardefassoc} generates equality between such parenthesizations with arbitrary permutations. We index the confinements and deconfinements using the transposition indices $1,2$ to simplify notation.

On the one hand, we have
\begin{equation}
\begin{split} (\vec{a} \otimes X^{\tau_1}_{\vec{b}}) \otimes X^{\tau_2}_{\vec{c}}= &  X^{\tau_1}_{c_{1}(\vec{a}) \otimes \vec{b}} \otimes X^{\tau_2}_{\vec{c}} \\ 
=& (d_{1}(c_{1}(\vec{a}) \otimes \vec{b}) \otimes d_{_2}(\vec{c})) \otimes (X^{\tau_1}_{\vec{1}} \otimes X^{\tau_2}_{\vec{1}}) \\
=& (d_{1}(c_{1}(\vec{a}) \otimes \vec{b}) \otimes d_{2}(\vec{c})) \otimes \begin{cases} X^{\tau_1\tau_2}_{\vec{1}} & \tau_1 \ne \tau_2 \\
\bigoplus_{c \in \Irr(\CC)} \cdots   \underbrace{c }_{i} \cdots \underbrace{c^* }_j  \cdots & \tau_1=\tau_2=(ij) \end{cases}
\end{split}
\end{equation} 

On the other hand, 
\begin{equation}\begin{split} \vec{a} \otimes (X^{\tau_1}_{\vec{b}} \otimes X^{\tau_2}_{\vec{c}})= &  \vec{a} \otimes \left (\left ( d_{1}(\vec{b}) \otimes d_{2}(\vec{c}) \right ) \otimes \left (X^{\tau_1}_{1} \otimes X^{\tau_2}_{1} \right)\right )\\ 
=& \left ( \vec{a} \otimes  d_{1}(\vec{b})  \otimes d_{2}(\vec{c}) \right) \otimes \begin{cases} 
X^{\tau_1\tau_2}_{\vec{1}} & \tau_1 \ne \tau_2 \\
\bigoplus_{c \in \Irr(\CC)} \cdots   \underbrace{c }_{i} \cdots \underbrace{c^* }_j  \cdots & \tau_1=\tau_2=(ij)
\end{cases}
\end{split}
\end{equation}

Now it is not necessarily the case that 
$d_{1}(c_{1}(\vec{a}) \otimes \vec{b}) \otimes d_{2}(\vec{c})) = \vec{a} \otimes d_{1}(\vec{b}) \otimes d_{2}(\vec{c})$. However, in the case $\tau_1\ne \tau_2$, it suffices to show that their confinements with respect to $\tau_1\tau_2$ are equal. 
 
 If $\tau_1 \cap \tau_2=\emptyset$, 
The confinement of $d_{1}(c_{1}(\vec{a}) \otimes \vec{b}) \otimes d_{2}(\vec{c}))$
is the object $\boxtimes_i X_i$
$$X_i =\begin{cases}  
a_i \otimes b_i \otimes c_i & i \notin \tau_1,\tau_2\\

a_{i_1} \otimes a_{j_1} \otimes b_i \otimes c_{i_1} \otimes c_{j_1} & i \in \tau_1  \\
 a_{j_1} \otimes a_{j_2} \otimes b_{i_2} \otimes b_{j_2} \otimes c_i & i \in \tau_2 \\

\end{cases},$$ which one can check this is the same as the confinement of  $\vec{a} \otimes  d_{1}(\vec{b})  \otimes d_{2}(\vec{c})$. The case where $\tau_1$ and $\tau_2$ multiply to a 3-cycle is similar.

Now let $\tau_1=\tau_2$, and write $\tau=(ij)$. Write
\begin{eqnarray}
d_{\tau}(\vec{b})_{i,j} = d_{i,j}\\
 d_{\tau}(\vec{c})_{i,j} = e_{i,j}\\
 d_{\tau}(c_{\tau}(\vec{a}) \otimes \vec{b})_{i,j}=f_{i,j} 
\end{eqnarray}
where the equations \begin{eqnarray} \label{eqns:deconfs1}d_i \otimes d_j = b_i=b_j \\ \label{eqns:deconfs2}e_i \otimes e_j = c_i = c_j \\ \label{eqns:deconfs3}f_i \otimes f_j = a_i \otimes a_j \otimes b_i = a_i \otimes a_j=b_j\end{eqnarray} are satisfied by Definition \ref{def:deconf}.  

Comparising the products layer-wise indexing by $k$ one has

$$(d_{\tau}(c_{\tau}(\vec{a}) \otimes \vec{b}) \otimes d_{\tau}(\vec{c})) \otimes 
\bigoplus_{c \in \Irr(\CC)} \cdots   \underbrace{c }_{i} \cdots \underbrace{c^* }_j  \cdots$$
and 
 $$\left ( \vec{a} \otimes  d_{\tau}(\vec{b})  \otimes d_{\tau}(\vec{c}) \right) \otimes 
\bigoplus_{c \in \Irr(\CC)} \cdots   \underbrace{c }_{i} \cdots \underbrace{c^* }_j  \cdots.$$

Clearly when $k \ne i,j$ the $k$th entry of both products are equal and hence it suffices to consider the $i$th and $j$th entries, which become 

$$\bigoplus_{c \in \Irr(\CC)} \cdots   \underbrace{ f_i \otimes e_i \otimes c }_{i} \cdots \underbrace{f_j \otimes e_j \otimes c^* }_j  \cdots$$
and 

$$\bigoplus_{c \in \Irr(\CC)} \cdots   \underbrace{a_i \otimes d_i \otimes e_i  \otimes c }_{i} \cdots \underbrace{a_j \otimes d_j \otimes e_j \otimes c^* }_j  \cdot$$

By Proposition \ref{prop:transpositionprod} and commutativity of multiplication in $C_{\id}$, checking whether these two objects are equal can be reduced to checking whether

$$ f_i \otimes f_j \otimes e_i \otimes e_j = a_i \otimes a_j \otimes d_i \otimes d_j \otimes e_i  \otimes e_j$$

Both sides simplify to $a_i \otimes a_j \otimes b_i \otimes c_i$ by Equations \ref{eqns:deconfs1} - \ref{eqns:deconfs3}.

\end{proof}

Together with the confinement map that determines anyon-defect fusion, the following proposition shows that bare transposition defects generate permutation defects in the same way that transpositions generate all of $S_n$. 
\begin{prop}
\label{prop:disjointdefprod}
\begin{enumerate}
\item Products of disjoint bare defects commute and are themselves bare defects.
\begin{eqnarray} X^{\rho}_{\vec{1}} \otimes X^{\sigma}_{\vec{1}} = X^{\rho\sigma}_{\vec{1}} & \text{ when } \rho \cap \sigma = \emptyset \end{eqnarray}
\item Every bare defect $X^{\sigma}_{\vec{1}}$ can be written as a product of bare transposition defects.
\end{enumerate}

\end{prop}
\begin{proof}
The first is an immediate consequence of Definition \ref{def:defectfusion}. The second follows from Definition \ref{def:defectfusion} and Proposition \ref{lem:bardefassoc} that one can write
\begin{equation}
X^{\sigma}_{\vec{1}}= \bigotimes_{i} X^{\sigma_i}_{\vec{1}} = \bigotimes_{i,j} X^{\tau^i_j}_{\vec{1}}
\end{equation}
where $\sigma_i=\tau^i_1\tau^i_2 \cdots \tau^i_{m_i}$ is any transposition decomposition of the $i$th disjoint cycle in the decomposition of $\sigma$. 
\end{proof}




\subsubsection{$S_n$-crossed braiding}

\begin{lem}
\label{lem:permcross2} The multiplication is $S_n$-crossed
$$X^{\rho}_{\vec{a}} \otimes X^{\sigma}_{\vec{b}}=\rho\cdot X^{\sigma}_{\vec{b}} \otimes X^{\rho}_{\vec{a}}.$$
\end{lem}
\begin{proof}
First observe that if $d_{\sigma}(\vec{b})$ is a $\sigma$-deconfinement of $X^{\sigma}_{\vec{b}}$, then $\rho\cdot d_{\sigma}(\vec{b})$ is a $\rho\sigma\rho^{-1}$-deconfinement of $X^{\rho\sigma\rho^{-1}}_{\rho\cdot \vec{b}}$. Since the fusion products are independent of choice of deconfinements, we have

\begin{equation}
\begin{split}
\rho\cdot X^{\sigma}_{\vec{b}} \otimes X^{\rho}_{\vec{a}} =& X^{\rho\sigma\rho^{-1}}_{\rho\cdot \vec{b}} \otimes X^{\rho}_{\vec{a}} \\
=& d_{\rho\sigma\rho^{-1}}(\rho\cdot \vec{b}) \otimes X^{\rho\sigma\rho}_{\vec{1}} \otimes X^{\rho}_{\vec{a}} \\
=& \rho\cdot d_{\sigma}(\vec{b}) \otimes  X^{\rho\sigma\rho^{-1}}_{\vec{1}} \otimes X^{\rho}_{\vec{a}} \\
=& \rho\cdot d_{\sigma}(\vec{b}) \otimes X^{\rho}_{\vec{a}} \otimes X^{\sigma}_{\vec{1}} \\
=& X^{\rho}_{\vec{a}} \otimes d_{\sigma}(\vec{b}) \otimes X^{\sigma}_{\vec{1}}\\
=& X^{\rho}_{\vec{a}} \otimes X^{\sigma}_{\vec{b}}
\end{split}
\end{equation}
\end{proof}

\begin{thm} 
\label{thm:fusring}The permutation defect fusion ring $(C^{\times}_{S_n}, \otimes)$ is a unital based $\ZZ_+$-ring which is $S_n$-crossed commutative. 
\end{thm}

\begin{proof}
 Let $A$ be the free $\ZZ$-module generated by the set $\left \{\Irr \left ( \CC^{\boxtimes n} \right) \bigcup_{\sigma} \Irr \left ( \CC_{\sigma} \right ) \right \}.$ Combining the multiplication coming from the fusion ring of $\CC^{\boxtimes n}$, the confinement map in Definition 5.2, and the definition of bare defect fusion products in Definitions 5.4 and 5.5 defines an associative binary operation $A \times A \longrightarrow A$ that gives $A$ the structure of a ring. 

As a consequence of our definitions the products are all given by non-negative integer linear combinations of basis elements $\{\Irr \left ( \CC^{\boxtimes n} \right) \bigcup_{\sigma} \Irr \left ( \CC_{\sigma} \right )\}$, and the unit $1 \in \Irr(\CC^{\boxtimes n})$ is simple. Therefore we have a unital $\ZZ_+$ ring. 

Now since $a\cong a^{**}$ for all simple objects in an MTC by pivotality, and $(g^{-1})^{-1}=g$ for all $g \in G$, the following map defines an involution on basis elements,
\begin{eqnarray}
\vec{a} \mapsto \vec{a}^* & \vec{a} \in \Irr(\CC^{\boxtimes n}) \\
X^{\sigma}_{\vec{a}} \mapsto X^{\sigma^{-1}}_{\vec{a}*}&  X^{\sigma}_{\vec{a}} \in \Irr(\CC_{\sigma}) 
\end{eqnarray}
which we expand linearly to an involution on all of $A$. Since that $N^{ab}_1=1$ if and only if $a \cong b^*$, it remains only to check that this involution corresponds to duality of defects.

Fix $\sigma \in S_n$, $\vec{a} \in \Irr(\CC^{\boxtimes n})$ and let $X^{\sigma}_{\vec{a}} \in \CC_{\sigma}$, $X^{\tau}_{\vec{b}} \in \CC_{\tau}$. Clearly $N^{X^{\sigma}_{\vec{a}}X^{\tau}_{\vec{b}}}_1=0$ unless $\tau=\sigma^{-1}$. 

Given a transposition decomposition of $\sigma$, $\sigma=\tau_m\cdots t_1$ with $\tau_j=(j_1j_2)$, $j_1 < j_2$, we have
\begin{eqnarray}
\begin{split}
X^{\sigma}_{\vec{a}} \otimes X^{\sigma^{-1}}_{\vec{b}} = \left ( d(\vec{a}) \otimes d(\vec{b}) \right ) \otimes (X^{\sigma}_{\vec{1}} \otimes X^{\sigma^{-1}}_{\vec{1}}) \\
\\= \left (d(\vec{a}) \otimes d(\vec{b}) \right) \otimes \left (\bigotimes_{1 \le j \le m} \left ( \bigoplus_{c \in \Irr(\CC)} 1^{\boxtimes j_1-1} \boxtimes c \boxtimes 1^{j_2-j_1} \boxtimes c^* \boxtimes 1^{n-j_2}  \right) \right)  
\end{split}
\end{eqnarray}  

In the fusion product over $j$ there is one summand of $\vec{1}$, hence $N^{X^{\sigma}_{\vec{a}}X^{\sigma^{-1}}_{\vec{b}}}_1=1$ only if $N^{d(\vec{a}),d(\vec{b})}_1=1$. So we must have $d(\vec{b})\cong d(\vec{a})^*$. In other words, the deconfinements must be dual. 

Now reconfinement with an arbitrary bare defect gives
\begin{eqnarray} X^{\rho}_{\vec{1}} & =  d(\vec{a}) \otimes d(\vec{b}) \otimes X^{\rho}_{\vec{1}} & = X^{\rho}_{\vec{a} \otimes \vec{b}}
\end{eqnarray}
and hence $\vec{b} \cong \vec{a}^*$. Finally, Lemma \ref{lem:permcross2} gives $S_n$-crossed commutativity. 
\end{proof}

\subsection{Twisted fusion of permutation defects}
The fusion ring we have defined forms a basepoint for an $H^2(S_n,A^{\boxtimes n})$-torsor.

\begin{defn}[Twisted defect fusion] 
Let $\omega: S_n \times S_n \to A^{\boxtimes n}$ be a 2-cocycle for the permutation action of $S_n$ on $A^{\boxtimes n}$. Define
\begin{equation} X^{\rho}_{\vec{a}} \otimes_{\omega} X^{\sigma}_{\vec{b}} := \omega(\rho,\sigma) \otimes (X^{\rho}_{\vec{a}} \otimes X^{\sigma}_{\vec{b}} ).
\end{equation}
\end{defn} 

\begin{lem}The twisted fusion product $\otimes_{\omega}$ gives an $S_n$-crossed fusion ring structure on the defects for every 2-cocycle $\omega: S_n \times S_n \to A^{\boxtimes n}$. 
\end{lem}
\begin{proof}The twisted fusion product $\otimes_{\omega}$ is associative:

On the one hand,
\begin{eqnarray}(X^{\rho}_{\vec{a}} \otimes_{\omega} X^{\sigma}_{\vec{b}}) \otimes_{\omega}  X^{\tau}_{\vec{c}} =& ( \omega(\rho,\sigma) \otimes (X^{\rho}_{\vec{a}} \otimes X^{\sigma}_{\vec{b}} )) \otimes_{\omega}  X^{\tau}_{\vec{c}}\\
=& \omega(\rho,\sigma) \otimes (X^{\rho}_{\vec{a}} \otimes X^{\sigma}_{\vec{b}} )) \otimes_{\omega}  X^{\tau}_{\vec{c}}\\
=& (\omega(\rho,\sigma) \otimes \omega(\rho,\sigma) ) \otimes (X^{\rho}_{\vec{a}} \otimes X^{\sigma}_{\vec{b}}) \otimes X^{\tau}_{\vec{c}}.
\end{eqnarray}
On the other hand, 
\begin{eqnarray}X^{\rho}_{\vec{a}} \otimes_{\omega} (X^{\sigma}_{\vec{b}} \otimes_{\omega}  X^{\tau}_{\vec{c}} ) =&   X^{\rho}_{\vec{a}} \otimes_{\omega}\left ( \omega(\sigma,\tau) \otimes ( X^{\sigma}_{\vec{b}} \otimes  X^{\tau }_{\vec{c}} )\right)\\ =&  \rho \cdot \omega(\sigma,\tau) \otimes \left ( X^{\rho}_{\vec{a}} \otimes_{\omega} ( X^{\sigma}_{\vec{b}} \otimes  X^{\tau }_{\vec{c}} )\right) \\
=& \rho \cdot \omega(\sigma,\tau) \otimes \omega( \rho, \sigma\tau) \otimes \left ( X^{\rho}_{\vec{a}} \otimes ( X^{\sigma}_{\vec{b} }\otimes X^{\tau}_{\vec{c}})\right) 
\end{eqnarray}
where we have used the $S_n$-crossed braiding and associativity of anyon-defect fusion.
Since $\omega$ is a 2-cocycle for the $S_n$-action on $A^{\boxtimes n}$, it satisfies
\begin{equation}
 \omega(\rho,\sigma) \otimes \omega(\rho,\sigma)  = \rho \cdot \omega(\sigma,\tau) \otimes \omega(\rho, \sigma\tau) 
\end{equation} 
\end{proof}

Whenever two 2-cocycles $\omega$ and $\tilde{\omega}$ differ by a 2-coboundary, their corresponding twisted fusion products $\otimes_{\omega}$ and $\otimes_{\tilde{\omega}}$ give rise to isomorphic $S_n$-crossed fusion rings.

\begin{lem} $(C^{\times}_{S_n}, \otimes_{\omega})$ and $(C^{\times}_{S_n}, \otimes_{\tilde{\omega}})$ are isomorphic as $S_n$-crossed fusion rings if and only if $\omega$ and $\tilde{\omega}$ differ by a 2-coboundary. \end{lem}

\begin{proof}
 there exists a map $\phi: S_n \to A$ such that
where $\phi$ satisfies
\begin{equation}
\omega(\rho,\sigma)\tilde{\omega}(\rho)^* = \rho \cdot \phi(\sigma) \otimes \phi(\rho\sigma)^* \otimes \phi(\rho).
\end{equation}

We check that $\phi$ defines a ring isomorphism $\Phi$ between the rings $(C^{\times}_{S_n}, \otimes_{\omega})$ and $(C^{\times}_{S_n}, \otimes_{\tilde{\omega}})$ if and only if $\omega$ and $\tilde{\omega}$ differ by a 2-coboundary. 
\begin{equation}
\Phi(X^{\rho}_{\vec{a}}) = \phi(\rho) \otimes X^{\rho}_{\vec{a}}
\end{equation}

On the one hand, 
\begin{eqnarray}
\Phi(X^{\rho}_{\vec{a}} \otimes_{\omega} X^{\sigma}_{\vec{b}}) =& \phi(\rho\sigma)\otimes \omega(\rho,\sigma)   \otimes (X^{\rho}_{\vec{a}} \otimes X^{\sigma}_{\vec{b}} ) .
\end{eqnarray}
On the other hand, 
\begin{equation}
\begin{split}
\Phi(X^{\rho}_{\vec{a}}) \otimes_{\tilde{\omega}} \Phi(X^{\sigma}_{\vec{b}}) =& (\phi(\rho) \otimes X^{\rho}_{\vec{a}} ) \otimes_{\tilde{\omega}} ( \phi(\sigma) \otimes X^{\sigma}_{\vec{b}} )\\
=&  (\phi(\rho) \otimes \rho \cdot \phi(\sigma) ) \otimes (X^{\rho}_{\vec{a}} \otimes_{\tilde{\omega}} X^{\sigma}_{\vec{b}}) \\
 =& (\phi(\rho) \otimes \rho \cdot \phi(\sigma) \otimes \tilde{\omega}(\rho,\sigma) \otimes (X^{\rho}_{\vec{a}} \otimes X^{\sigma}_{\vec{b}}) \\
 =&  \phi(\rho\sigma)\otimes \omega(\rho,\sigma)   \otimes (X^{\rho}_{\vec{a}} \otimes X^{\sigma}_{\vec{b}} )
 \end{split}
\end{equation}
\end{proof}

\subsection{The permutation defect fusion ring and algorithm}
We have constructed a torsor of $S_n$-crossed fusion rings over $H^2(S_n,A^{\boxtimes n})$, which classify such fusion rings extending those $\CC^{\boxtimes n}$. Since the $H^4(S_n,U(1))$ obstruction vanishes by the work of \cite{GJ}, every such $S_n$-crossed extension ring lifts to an $S_n$-crossed braided extension \cite{ENO}. Therefore we can conclude that the twisted $S_n$-defect fusion construction realizes the fusion rules for permutation extensions of MTCs $\CC^{\boxtimes n}$. 

\begin{thm}  Given an MTC $\CC$ and a 2-cocycle $\omega$ representing $[\omega] \in H^2(S_n,A^{\boxtimes n})$, the equivalence classes of fusion rings of $S_n$-crossed braided extensions of $\CC^{\boxtimes n}$ are given by the isomorphism classes of the fusion rings $(C^{\times}_{S_n},\otimes_{\omega})$.
\end{thm}
  
In particular, the following Algorithm \ref{alg:fusion} produces the correct fusion rules. 

  \begin{framed}
  \begin{alg}[Permutation defect fusion algorithm]
  \label{alg:fusion}
  The fusion product of two permutation defects $X^{\sigma}_{\vec{a}}$ and $X^{\tau}_{\vec{b}}$ can be computed as follows. 
  \begin{enumerate}
  \item $\sigma$- and $\tau$-deconfinement: \\ Strip the topological charges from the defects and twist with the abelian anyon $\omega(\sigma,\tau)$.
  \item Transposition defect annihilation:\\ Compute the fusion product of the bare $\sigma$- and $\tau$-defects, for every pair of indices $(ij)$ permuted by both $\sigma$ and $\tau$, pulling out a factor of $$\bigoplus_{c \in \Irr(\CC)} \cdots \underbrace{c}_{i} \cdots \underbrace{c^*}_{j} \cdots.$$
  
  \item $\sigma\tau$-confinement:\\ Confine the product of the objects from Step 1 and Step 2 with the bare $\sigma\tau$-defect.
  \end{enumerate}
  \end{alg}
  \end{framed}
  An implementation of the algorithm is given in a Mathematica \textbf{PermutationDefectFusion.m} package. In Appendix \ref{app:code} we show the main function implementing Algorithm \ref{alg:fusion}; the full code is available at the author's website \url{math.ucsb.edu/~cdelaney/research}. 

\section{Examples}
The salient features of the $S_n$-crossed fusion ring construction can be illustrated through a few small examples. First we recall the bilayer symmetry defect fusion rules established in \cite{BBCW,EMJP,BarmeierSchweigert} from which the general permutation defect fusion rules are built.  The fusion rules for the trilayer Fibonacci defects of the rank 24 fusion category $\left(\Fib^{\boxtimes 3}\right)^{\times}_{S_3}$  in Section \ref{examp:trilayerfib} exhibit the overall pattern of $S_n$-crossed fusion rings, in particular how they are determined by the fusion ring of $\CC$ and bare $S_2$-defects.  In Section \ref{examp:trilayerising} we repeat the example for the rank 3 $\Ising$ MTCs with fusion rules
$$\begin{cases} \sigma \otimes \sigma = 1 \oplus \psi \\
\sigma \otimes \psi=\psi \otimes \sigma=1 \end{cases},$$ 
and give the untwisted fusion rules that form a basepoint for the $H^2(S_3, A^{\boxtimes 3})$-torsor of possible fusion rules of $\left(\Ising^{\boxtimes 3}\right)^{\times}_{S_3}$ with 
$$A^{\boxtimes 3} \cong (\ZZ/2\ZZ)^3 = \langle 11\psi, 1\psi1, \psi 1 1\rangle .$$ 

\subsection{Bilayer topological order with $\mathds{Z}/2\mathds{Z}$-symmetry}
$$\begin{tikzpicture}[scale=.5,baseline=0]
\draw[thick,fill=gray!15] (.5,.5)--(0,0)--(4,0)--(4.5,.5) node[right] {$\CC$}--(.5,.5);
\draw[thick,fill=gray!15] (.5,-.5)--(0,-1)--(4,-1)--(4.5,-.5) node[right] {$\CC$}--(.5,-.5);
\draw (-.5,0) node[left] {$S_2 \curvearrowright$} ;
\draw[thick, ->] (6,0)--(7,0) node[right] {$\left ( \CC^{\boxtimes 2} \right)^{\times}_{S_2}=\CC^{\boxtimes 2} \oplus \CC_{(12)}$};
\end{tikzpicture} $$
Fusion rules for $S_2$-extensions of $\CC \boxtimes \CC$ were first proven \cite{BarmeierSchweigert} in the language of modular functors. More recently they appeared in the physics literature in \cite{BBCW} and in the case of $ ( \Fib^{\boxtimes 2} )^{\times}_{\ZZ/2\ZZ}$ were computed directly by computer assisted de-equivariantization of the MTC $SU(2)_8$ \cite{CGPW}. Fusion rules for $\left (\CC \boxtimes \CC \right)^{\times}_{S_2}$ directly in terms of $\CC$ were then proven in \cite{EMJP}. 

Recasting these results yet again in our model, the form of the general fusion rules for 
$$\left (\CC \boxtimes \CC \right)^{\times}_{S_2}= \CC_{\id} \oplus \CC_{(12)}$$ can be understood as follows.

Writing $\Irr(\CC_{(12)})=\{X^{(12)}_{a\boxtimes a} | a \in \Irr(\CC)\}$, for the fusion rules corresponding to the trivial cocycle $\omega \equiv 1$ we find that every defect is fixed by the $S_2$ action and related to the bare defect by fusion with a monolayer anyon $a \boxtimes 1$ (also $1 \boxtimes a$). 
In particular the fusion rules can be expressed as
$$\begin{cases}
a\boxtimes b \otimes X^{(12)}_{c\boxtimes c} = X^{(12)}_{a \otimes b \otimes c \boxtimes a \otimes b \otimes c} = \bigoplus_d N^{abc}_d X_{d\boxtimes d}\\
X^{(12)}_{a\boxtimes a} \otimes X^{(12)}_{b\boxtimes b} = \bigoplus_{c \in \Irr(\CC)} a \otimes c \boxtimes c^* \otimes b \end{cases}$$
where $N^{abc}_d$ are generalized fusion coefficients in $\CC$.
 
The quantum dimensions satisfy
\begin{eqnarray}
d_{X_{11}}= \sqrt{\frac{\DD^2_{\CC}}{\sum_{c \in \Irr(\CC)} d_c^2}}\\
d_{X_{aa}}= d_a \cdot  \sqrt{\frac{\DD^2_{\CC}}{\sum_{c \in \Irr(\CC)} d_c^2}}
\end{eqnarray}
\subsubsection{Example: Bilayer Fibonacci with $S_2$-symmetry}
$$\begin{tikzpicture}[scale=.5,baseline=0]
\draw[thick,fill=red!15] (.5,.5)--(0,0)--(4,0)--(4.5,.5) node[right] {$\Fib$}--(.5,.5);
\draw[thick,fill=red!15] (.5,-.5)--(0,-1)--(4,-1)--(4.5,-.5) node[right] {$\Fib$}--(.5,-.5);
\draw (-1.5,0) node[left] {$S_2$} ;
\draw[thick,looseness=3,<->] (-.5,-.75) to [out=180, in=180] (-.5,.25);

\draw[thick, ->] (6,0)--(7,0) node[right] {$\left ( \Fib^{\boxtimes 2} \right)^{\times}_{S_2}=\Fib^{\boxtimes 2} \oplus \CC_{(12)}$};
\end{tikzpicture} $$
\label{examp:bilayerfib}
The Fibonacci MTC is rank 2 and has nontrivial fusion rule $\tau \otimes \tau = 1 \oplus \tau$. Writing the bilayer anyons $a\boxtimes b =: ab$ and labeling defects by fixed points, we put
 $$ \Irr \left ( \left ( \Fib^{\boxtimes 2} \right)^{\times}_{S_2} \right )=\Irr \left ( \Fib^{\boxtimes 2} \right) \cup \Irr \left ( \CC_{(12)} \right )=\{11,1\tau,\tau 1,\tau\tau, X_{11}, X_{\tau\tau}\}.$$

The fusion rules constructed in the previous section in the case $\CC=\Fib$ and $n=2$ are listed in the table below.

\begin{table}[h]
$$\small\begin{array}{c|c|c|c|c||c|c}
\otimes & 11 &1\tau & \tau 1 & \tau\tau & X_{11} & X_{\tau\tau}\\
\thickhline 
11 &11 & 1\tau & \tau 1 & \tau\tau & X_{11} & X_{\tau\tau} \\ \hline
1\tau& 1\tau  & 11 \oplus 1\tau & \tau\tau & \tau 1\oplus \tau\tau & X_{\tau\tau} & X_{11} \oplus X_{\tau\tau} \\ \hline
\tau 1 & \tau 1 & \tau\tau & 11 \oplus \tau 1 & 1\tau \oplus \tau\tau & X_{\tau\tau} & X_{11} \oplus X_{\tau\tau} \\\hline
\tau\tau & \tau \tau & \tau1 \oplus \tau\tau & 1\tau \oplus \tau\tau &  11 \oplus 1\tau \oplus \tau 1 \oplus \tau \tau & X_{11} \oplus X_{\tau\tau} &  X_{11} \oplus 2X_{\tau\tau} \\
\hline\hline X_{11} & X_{11} & X_{\tau\tau} & X_{\tau\tau} & X_{11} \oplus X_{\tau\tau} & 11 \oplus \tau\tau & 1 \tau \oplus \tau 1 \oplus \tau \tau  \\
\hline 
X_{\tau\tau} & X_{\tau\tau} & X_{11} \oplus X_{\tau\tau} & X_{11} \oplus X_{\tau\tau} & X_{11} \oplus 2 X_{\tau\tau} & 1 \tau \oplus \tau 1 \oplus \tau \tau & 11 \oplus 1 \tau \oplus \tau 1 \oplus 2 \tau \tau \\
\end{array}$$
\caption{Fusion table for bilayer Fibonacci anyons and defects.}
\end{table}

\subsection{Trilayer topological order with $S_3$-symmetry}

The general pattern of fusion in $S_n$-extensions can be seen from the $n=3$ case. Ranks of permutation extensions are large even for small $n$, with $\text{rank}\left((\Fib^{\boxtimes 3})^{\times}_{S_3}\right)=24$ and $\text{rank}\left((\Ising^{\boxtimes^3})^{\times}_{S_3}\right)=60$. So rather than providing the full fusion table for these examples we give a more compact description in the form of the fusion table for bare permutation defects and show how to use them to compute arbitrary products. For general $n$, the fusion rules can be derived from the $\frac{n(n-1)}{2} \times \frac{n(n-1)}{2}$ table of its bare transposition defects, but even then all transpositions behave identically and can be just as well understood through the $n=2$ case and the fusion algorithm. 
\subsubsection{Example: Trilayer Fibonacci with $S_3$-symmetry}
\label{examp:trilayerfib}

$$\begin{tikzpicture}[scale=.5,baseline=-32.5]
\draw[thick,fill=red!15] (.5,.5)--(0,0)--(4,0)--(4.5,.5) node[right] {$\Fib$}--(.5,.5);
\draw[thick,fill=red!15] (.5,-.5)--(0,-1)--(4,-1)--(4.5,-.5) node[right] {$\Fib$}--(.5,-.5);
\draw[thick,fill=red!15] (.5,-1.5)--(0,-2)--(4,-2)--(4.5,-1.5) node[right] {$\Fib$} --(.5,-1.5);
\draw (-.5,-.75) node[left] {$S_3$} ;

\draw[thick, ->] (6,-.75)--(7,-.75) node[right] {$\left ( \Fib^{\boxtimes 3} \right)^{\times}_{S_3}$};
\end{tikzpicture} $$

We write the $3!=6$ graded components as
 $$\left ( \Fib^{\boxtimes 3} \right)^{\times}_{S_3}=\Fib^{\boxtimes 3} \oplus \CC_{(12)} \oplus \CC_{(23)} \oplus \CC_{(13)} \oplus \CC_{(123)} \oplus \CC_{(132)}.$$
Suppressing the $\boxtimes$ in the label set $\{1,\tau \}{^\boxtimes 3}$ we label the defects in each sector $\CC_{\sigma}$ by the anyons fixed under the permutation of $\sigma$. We count a rank 24 fusion category where each $\CC_{\sigma}$ has global quantum dimension \begin{eqnarray}\DD_{\CC_{\sigma}}=\DD_{\Fib^{\boxtimes 3}}=&\DD_{\Fib}^3 = (2+\phi)^{3/2}&=\sqrt{15+20\phi}. \end{eqnarray} 

\begin{table}[h!]
 \begin{center}
 \scalebox{.85}{
 \begin{tabular}{ccc}
$\sigma$-sector & $\sigma$-defects & quantum dim.\\ 
\thickhline
$\CC_{\id}$ & $ \{111,11\tau, 1\tau 1, 1 \tau\tau, \tau 11, \tau 1 \tau, \tau \tau 1, \tau\tau\tau\}$ & $\{1,\phi,\phi,\phi^2,\phi,\phi^2,\phi^2,\phi^3\}$ \\
\hline
$\CC_{(12)}$ & $\{X_{111},\, X_{ 1 1\tau},\, X_{\tau \tau 1},\, X_{\tau\tau\tau}\}$ & $\{\sqrt{2+\phi}, \sqrt{3+4\phi}, \sqrt{3+4\phi}, \sqrt{7+11\phi} \}$ \\ 
\hline
$\CC_{(23)}$ & $\{Y_{111},\, Y_{\tau 1 1},\, Y_{1 \tau \tau },\, Y_{\tau\tau\tau}\}$ & $\{\sqrt{2+\phi}, \sqrt{3+4\phi}, \sqrt{3+4\phi}, \sqrt{7+11\phi} \}$\\ 
\hline
$\CC_{(13)}$ & $\{Z_{111},\, Z_{1\tau  1},\, Z_{ \tau 1\tau},\, Z_{\tau\tau\tau}\}$ & $\{\sqrt{2+\phi}, \sqrt{3+4\phi}, \sqrt{3+4\phi}, \sqrt{7+11\phi} \}$\\
\hline
$\CC_{(123)}$ & $\{U_{111}, \, U_{\tau\tau\tau}\}$ & $\{\sqrt{5+5\phi}, \sqrt{10+15\phi}\}$ \\
\hline
$\CC_{(132)}$ & $\{V_{111}, \, V_{\tau\tau\tau}\}$ & $\{\sqrt{5+5\phi},\sqrt{10+15\phi}\}$\\  
\end{tabular}}
\end{center}
\caption{Quantum dimensions of simple objects in the invertible $\Fib^{\boxtimes 3}$-bimodule categories $\CC_{\sigma}$.}
\end{table}

The Fibonacci MTC has no nontrivial abelian anyons, so $A=1$ and $H^3(S_3,A)$ is trivial. By the existence result of \cite{GJ} and the classification of $G$-crossed braided extensions of BFCs \cite{ENO}, there is a unique fusion ring shared among $S_3$-crossed extensions of $\Fib^{\boxtimes 3}$.

We list the fusion between anyons and bare defects in Table \ref{table:trilayerfib}.

\begin{table}[h!]
\label{table:trilayerfib}
$$
\small
\begin{array}{c||c|c|c|c|c}
\otimes & X_{111} & Y_{111} & Z_{111} & U_{111}  & V_{111} \\ 
\thickhline
11\tau  &X_{11\tau}&Y_{1 \ta\ta}&Z_{\tau 1 \tau} & U_{\ta\ta\ta}&V_{\ta\ta\ta}  \\
1\tau 1  & X_{\ta\tau 1}&Y_{1\tau\tau}& Z_{1\tau 1} & U_{\ta\ta\ta}&V_{\ta\ta\ta} \\
\tau 1 1 &  X_{\tau \tau 1}& Y_{\tau 1 1}& Z_{\tau 1 \tau}& U_{\ta\ta\ta}&V_{\ta\ta\ta}\\
1\ta\ta &X_{\tau\tau\tau} & Y_{111}\oplus Y_{1\ta\ta} &Z_{\ta\ta\ta}&U_{111}\oplus U_{\ta\ta\ta}&V_{111}\oplus V_{\ta\ta\ta}\\
\ta 1 \tau &X_{\tau\tau\tau} & Y_{\ta\ta\ta} &Z_{111}\oplus Z_{\ta 1 \ta}&U_{111}\oplus U_{\ta\ta\ta}&V_{111}\oplus V_{\ta\ta\ta}\\

\tau \tau 1  &X_{111}\oplus X_{\ta\ta1} & Y_{\ta\ta\ta} &Z_{\ta\ta\ta}&U_{111}\oplus U_{\ta\ta\ta}&V_{111}\oplus V_{\ta\ta\ta}\\

\tau\tau\tau & X_{11\ta} \oplus  X_{\ta\ta \ta} &Y_{\ta11} \oplus Y_{\ta\ta\ta}& Z_{1\ta1} \oplus Z_{\ta \ta \ta} & U_{111} \oplus 2 U_{\ta\ta\ta} & V_{111} \oplus 2V_{\ta\ta\ta} \\
  \hline
\hline
  X_{111} &111 \oplus \tau\tau1&U_{111}&V_{111}&Y_{111} \oplus Y_{\tau\tau\tau} & Z_{111}\oplus Z_{\tau \tau \tau}\\
\hline Y_{111} & V_{111} &111 \oplus 1\ta\ta &U_{111}&Z_{111}\oplus Z_{\ta\ta\ta} & X_{111}\oplus X_{\ta\ta\ta} \\
\hline Z_{111}  &U_{111} &V_{111} &111 \oplus \ta1\ta&X_{111} \oplus X_{\ta\ta\ta} & Y_{111} \oplus Y_{\ta\ta\ta} \\
\hline  U_{111}& Z_{111} \oplus Z_{\ta\ta\ta} & X_{111} \oplus X_{\ta\ta\ta} & Y_{111} \oplus Y_{\ta\ta\ta} &2V_{111} \oplus V_{\tau\tau\tau} &111\oplus \ta\ta1 \oplus \ta 1 \ta \oplus 1 \ta \ta \oplus \ta \ta \ta  \\

\hline  V_{111} & Y_{111}\oplus Y_{\ta\ta\ta}& Z_{111}\oplus Z_{\ta\ta\ta} &X_{111}\oplus X_{\ta\ta\ta} & 111\oplus \ta\ta1 \oplus \ta 1 \ta \oplus 1 \ta \ta \oplus \ta \ta \ta &2U_{111} \oplus U_{\tau\tau\tau}  \\
\end{array}$$
\caption{Fusion rules for bare defects in $\left(\Fib^{\boxtimes 3}\right)^{\times}_{S_3}$. Our convention is that $X^{\rho}_{\vec{1}} \otimes X^{\sigma}_{\vec{1}}\in \CC^{\rho\sigma}$ corresponds to the entry in the $\rho$-row and  $\sigma$-column. }
\end{table}
By Section \ref{sec:permdefring}, the full fusion table for the $S_3$-extension is determined by rows 1-3 and 8-10.

To find the fusion rule between two arbitrary defects from the table, make any choice of deconfinements $d_{\rho}(\vec{a})$ and $d_{\sigma}(\vec{b})$, and confine their product $d_{\rho}(\vec{a}) \otimes d_{\sigma}(\vec{b})$ with the entry of the table corresponding to the bare product $X^{\rho}_{\vec{1}} \otimes X^{\sigma}_{\vec{1}}$.

For example, to compute 
$X^{(12)}_{\tau\tau1} \otimes X^{(132)}_{\tau\tau\tau}$ choose deconfinements
\begin{eqnarray} \ta 11 \otimes X^{(12)}_{\vec{1}} = X^{(12)}_{\ta\ta1} \\
 11\tau \otimes X^{(132)}_{\vec{1}} =  X^{(132)}_{\tau\tau\tau} 
 \end{eqnarray} and write
\begin{eqnarray}
X^{(12)}_{\tau\tau1} \otimes X^{(132)}_{\tau\tau\tau} =&( \tau 1 1 \otimes 11\tau) \otimes (X^{(12)}_{\vec{1}} \otimes X^{(132)}_{\vec{1}} ) \\
=& \tau 1 \tau \otimes (X^{(13)}_{\vec{1}} \oplus X^{(13)}_{\tau 1 \tau} )\\
=& X^{(13)}_{\tau \otimes \tau, 1, \tau \otimes \tau} \oplus X^{(13)}_{\tau \otimes \tau \otimes \tau, 1 , \tau \otimes \tau \otimes \tau} \\
=& (X^{(13)}_{111} \oplus X^{(13)}_{\ta 1 \ta}) \oplus (X^{(13)}_{111} \oplus 2X^{(13)}_{\ta 1 \ta} ) \\
=& 2 X^{(13)}_{111} \oplus 3  X^{(13)}_{\ta 1 \ta}.
\end{eqnarray}
The fusion product between any two defect types can be computed in this manner. 

\subsubsection{Trilayer Ising with $S_3$-symmetry}
\label{examp:trilayerising}
$$\begin{tikzpicture}[scale=.5,baseline=-32.5]
\draw[thick,fill=blue!15] (.5,.5)--(0,0)--(4,0)--(4.5,.5) node[right] {$\Ising$}--(.5,.5);
\draw[thick,fill=blue!15] (.5,-.5)--(0,-1)--(4,-1)--(4.5,-.5) node[right] {$\Ising$}--(.5,-.5);
\draw[thick,fill=blue!15] (.5,-1.5)--(0,-2)--(4,-2)--(4.5,-1.5) node[right] {$\Ising$} --(.5,-1.5);
\draw (-.5,-.75) node[left] {$S_3$ $\curvearrowright$} ;
\draw[thick, ->] (7,-.75)--(8,-.75) node[right] {$\left ( \Ising^{\boxtimes 3} \right)^{\times}_{S_3}$};
\end{tikzpicture} $$
We calculate the fusion rules for $(\Ising^{\boxtimes 3})^{\times}_{S_3}$ corresponding to the trivial cohomology class in $H^2(S_3, A^{\boxtimes 3})$.

Table \ref{table:trilayerising} contains the fusion products of the bare permutation defects. For compactness, we have used the shorthand $\vec{1}=111,\, \vec{\sigma}=\sigma\sigma\sigma,\, \vec{\psi}=\psi\psi\psi$ in defect charge labels and $X^{\sigma}_{\vec{1}\oplus \vec{\sigma} \oplus \vec{\psi}}:= X^{\sigma}_{\vec{1}} \oplus X^{\sigma}_{\vec{\sigma}} \oplus X^{\sigma}_{\vec{\psi}}$. 

\begin{table}[h!]
\label{table:trilayerising}
$$\small
\begin{array}{c|c|c|c|c|c}
\otimes & X^{(12)}_{\vec{1}} & X^{(23)}_{\vec{1}} & X^{(13)}_{\vec{1}} & X^{(123)}_{\vec{1}} & X^{(132)}_{\vec{1}} \\
\thickhline
X^{(12)}_{\vec{1}} & 111 \oplus \sigma\sigma1\oplus\psi\psi 1 & X^{(132)}_{\vec{1}} & X^{(123)}_{\vec{1}} &   X^{(23)}_{\vec{1} \oplus \vec{\sigma} \oplus \vec{\psi}} & X^{(13)}_{\vec{1} \oplus \vec{\sigma} \oplus \vec{\psi}}  \\ \hline
X^{(13)}_{\vec{1}}  &X^{(123)}_{\vec{1}} &X^{(132)}_{\vec{1}}  & 111 \oplus 1 \sigma 1 \oplus 1 \psi 1 & X^{(12)}_{\vec{1} \oplus \vec{\sigma} \oplus \vec{\psi}}&  X^{(23)}_{\vec{1} \oplus \vec{\sigma} \oplus \vec{\psi}}  \\\hline
X^{(23)}_{\vec{1}}  &  X^{(132)}_{\vec{1}} & 111 \oplus 1\sigma\sigma \oplus 1\psi\psi &X^{(123)}_{\vec{1}} & X^{(13)}_{\vec{1} \oplus \vec{\sigma} \oplus \vec{\psi}} &  X^{(12)}_{\vec{1} \oplus \vec{\sigma} \oplus \vec{\psi}} \\\hline
X^{(123)}_{\vec{1}} & X^{(13)}_{\vec{1} \oplus \vec{\sigma} \oplus \vec{\psi}}&X^{(12)}_{\vec{1} \oplus \vec{\sigma} \oplus \vec{\psi}} & X^{(23)}_{\vec{1} \oplus \vec{\sigma} \oplus \vec{\psi}}& 3X^{(132)}_{\vec{1}} \oplus X^{(132)}_{\vec{\psi}} & 111 \\ 
& & & & &  \oplus\, 1\sigma\sigma \oplus \sigma 1 \sigma \oplus \sigma \sigma 1 \\
& & & & &  \oplus\,  1\psi\psi  \oplus\psi \psi 1  \oplus \psi 1 \psi  \\
& & & & &  \oplus \, \sigma\sigma \psi \oplus \sigma\psi\sigma\oplus \psi \sigma\sigma  \\\hline
X^{(132)}_{\vec{1}}& X^{(23)}_{\vec{1} \oplus \vec{\sigma} \oplus \vec{\psi}} & X^{(13)}_{\vec{1} \oplus \vec{\sigma} \oplus \vec{\psi}} &  X^{(12)}_{\vec{1} \oplus \vec{\sigma} \oplus \vec{\psi}}&111 & 3X^{(123)}_{\vec{1}} \oplus X^{(123)}_{\vec{\psi}}\\
& & & &\oplus \, 1\sigma\sigma \oplus \sigma 1 \sigma \oplus \sigma \sigma 1  &  \\
& & & &  \oplus \,   1\psi\psi  \oplus\psi \psi 1  \oplus \psi 1 \psi \oplus &  \\
& & & &  \oplus \,  \sigma\sigma \psi \oplus \sigma\psi\sigma\oplus \psi \sigma\sigma &  \\
\end{array}$$
\caption{Fusion rules for bare defects in $\left(\Ising^{\boxtimes 3}\right)^{\times}_{S_3}$.}
\end{table}

\section{On generalizations and applications}

The discussion of $g$-confinement and $g$-deconfinement in Section \ref{sec:theory} and the idea of the algorithm in Section \ref{sec:permdefring} can be applied to compute other fusion rings.

We conjecture that fusion rules between inverse defects in $G$xBFCs conform to the rule
 $$X^g_1 \otimes X^{g^{-1}}_1 = \bigoplus_{c} \, c \otimes g\cdot c^*$$
where $c$ sums over a minimal subset of $\Irr(\CC)$ which includes the identity, is closed under duality, and whose orbit under the tensor product and $g$-action generates all of $\Irr(\CC)$. 

The rough idea of how to generalize our algorithm given $G$ a finite group with presentation $G=\left \langle g_1, g_2, \ldots, g_k \big | r_1, r_2, \ldots, r_l\right \rangle$ on generators is as follows. By writing arbitrary $g, h \in G$ as words in generators $g= g_1g_2\cdots g_{\alpha}$, $h=h_1h_2 \cdots h_{\beta}$ , fusion of arbitrary defects proceeds by computing the product of $gg_{\alpha}^{-1}g_{\alpha}h$ and rewriting the group labels on the resulting bare defects using the group relations. Of course the last step involves some art to guarantee the algorithm terminates and thus isn't efficient in general, but for small groups it suffices to quickly determine the fusion ring of the $G$-extension.

In particular all small examples of $\ZZ/2\ZZ$-extensions of MTCs we checked could be described in this manner. We include two brief examples of untwisted fusion rules, the $\ZZ/2\ZZ$-toric code with electromagnetic duality symmetry and $\text{Vec}_{\ZZ/3\ZZ}$ anyons with charge-conjugation symmetry, see \cite{BBCW} or \cite{DW} for more details.

\begin{examp}[Toric code with $e\leftrightarrow m$ symmetry]

$$\begin{cases}
e \otimes X_1 = X_f\\ 
m \otimes X_1 = X_f\\
f \otimes X_1 = X_1 \\ 
e \otimes X_f = X_1 \\
m \otimes X_f =X_1\\
f \otimes X_f = X_f \\
X_1 \otimes X_1 = 1\oplus f \\
X_1 \otimes X_f = e \oplus m \\
X_f \otimes X_f = 1 \oplus f \\
\end{cases}$$

\end{examp}

\begin{examp}[$\ZZ_3$-anyons with charge-conjugation symmetry]
$$
\begin{cases}
\omega \otimes X_1 = X_1\\ 
\omega* \otimes X_1 = X_1\\ 
X_1 \otimes X_1 = 1 \oplus \omega \oplus \omega^* \\
\end{cases}$$
\end{examp}

We conclude with a few short comments on some related directions. 
\subsection{More general symmetries}
While our main result shows that it is still possible to learn new things about MTCs through elementary considerations, already there has been work towards a more general extension theory of MTCs by mathematical objects with richer structure than groups, for example the Hopf monads of \cite{CSW} or hypergroups of \cite{Bischoff2017}.
The success of a classical approach here suggests it may too be possible to deduce fusion rules for more general symmetry-enriched categories using only the decategorified part of the symmetry. 

\subsection{Topological phases of matter and quantum computing with anyons and defects}
Determining the possible quantum logical operations that can result from exchanging and measuring anyons and defects in an SET phase requires algebraic data beyond the fusion rules. However, the fusion rules dictate which units of quantum information e.g. qubits, qutrits, qudits, etc. can be encoded in the multi-fusion channels of objects. 

For example, it follows from Theorem \ref{thm:fusring} that a collection of 4 bare $\tau=(ij)$ defects with total vacuum charge in $\left ( \CC^{\boxtimes n} \right)^{\times}_{S_n}$ are always qudits with $d=\text{rank}(\CC)$.

In upcoming work with Eric Samperton we leverage the approach to constructing fusion rings given here to construct their categorifications and then apply it to derive algebraic data for SET phases, from which we can derive insights into the interplay of symmetry, topological order, and quantum information, see also \cite{Dissertation}.
\newpage
\appendix

\section{Permutation defect fusion code sample}
\label{app:code}
\begin{lstlisting}
PermDefectFusion[defsector1_, defcharge1_, defsector2_, defcharge2_] :=
   Module[{i, j, k, l, m, deconfinements, righthandsector, numdiscycles1, numdiscycles2, 
   cycle1, cycle2, permlist2, transpdecomp1, transp, newsector, transplist, index1pos, index2pos, productanyons, finaldefsector, finalanyonproduct, finalfusionproduct},
   
   If[defsector1 == {} && defsector2 == {}, 
    Return[AnyonFusion[defcharge1, defcharge2]]];
 
   If[defsector1 == {} && defsector2 != {}, 
    Return[ObjListToDefLabel[
      AnyonDefectFusion[defcharge1, defsector2, defcharge2], defsector2]]];
   If[defsector1 != {} && defsector2 == {}, 
    Return[AnyonDefectFusion[defcharge2, defsector1, defcharge1]]];
   
   deconfinements = MultilayerLabelToObj /@ 
   {MinimalDeconfinement[defsector1, defcharge1], MinimalDeconfinement[defsector2, defcharge2]};

   numdiscycles1 = Length[defsector1];
   numdiscycles2 = Length[defsector2];
   righthandsector = defsector2;
   
   For[i = 1, i <= numdiscycles1, i++,
    cycle1 = {defsector1[[i]]};
    permlist2 = defsector2 // Flatten;
    transpdecomp1 = TranspositionDecomposition[cycle1];

    	For[k = 1, k <= Length[transpdecomp1], k++,
     	transp = Cycles[{Reverse[transpdecomp1][[k]]}];
     	transplist = transp[[1, 1]];
     	newsector = PermutationProduct[transp, Cycles[righthandsector]][[1]];
     	righthandsector = newsector;
 
     	Which[Length[Intersection[transplist, permlist2]] <= 1,
      permlist2 = righthandsector // Flatten;
      ,
      Length[Intersection[transplist, permlist2]] == 2,
      
      index1pos = Position[defsector2, transplist[[1]] ][[1, 1]];
      index2pos = Position[defsector2, transplist[[2]] ][[1, 1]];
      If[index1pos != index2pos,
       permlist2 = righthandsector // Flatten,
       productanyons = BareTranspositionDefectFusion[transp[[1]]];
       AppendTo[deconfinements, productanyons];
       permlist2 = righthandsector // Flatten
       ];
      ];
     ];
    ];
  
   finaldefsector = righthandsector;
   finalanyonproduct = Fold[ObjTensor][deconfinements];
   finalfusionproduct = Confinement[finalanyonproduct, finaldefsector, Table[1, numlayers]];
   Return[finalfusionproduct];];
\end{lstlisting}

\end{document}